\documentclass[a4paper,final]{amsart}
\usepackage{amsfonts}
\usepackage{amssymb}
\usepackage[latin2]{inputenc}
\usepackage{enumerate}
\usepackage{amsmath}
\usepackage[notcite,notref]{showkeys}
\usepackage{amssymb}
\usepackage{comment}

\usepackage[active]{srcltx}  
\newtheorem{theorem}{Theorem}[section]
\newtheorem{lemma}[theorem]{Lemma}

\newtheorem{corollary}[theorem]{Corollary}

\newtheorem{prop}[theorem]{Proposition}

\theoremstyle{definition}
\newtheorem{remark}[theorem]{Remark}

\newcommand{\bt}{\begin{theorem}}
\newcommand{\et}{\end{theorem}}

\newcommand{\N}{\mathbb{N}}
\newcommand{\Z}{\mathbb{Z}}
\newcommand{\Q}{\mathbb{Q}}
\newcommand{\R}{\mathbb{R}}

\newcommand{\eps}{\varepsilon}
\newcommand{\de}{\delta}

\newcommand{\ii}{\mathbf{i}}

\def\beq{\begin{equation}}
\def\eeq{\end{equation}}
\newcommand{\floor}[1]{\lfloor #1 \rfloor}

\newcommand{\bp}{\begin{proof}}
\newcommand{\ep}{\end{proof}}

\DeclareMathOperator{\diam}{diam}

\DeclareMathOperator{\hdim}{dim_{\mathsf{H}}}
\DeclareMathOperator{\pdim}{dim_{\mathsf{P}}}
\DeclareMathOperator{\bdim}{dim_{\mathsf{B}}}
\DeclareMathOperator{\ubdim}{{\overline{dim}_{\mathsf{B}}}}
\DeclareMathOperator{\mubdim}{{\overline{dim}_{\mathsf{MB}}}}
\DeclareMathOperator{\lbdim}{{\underline{dim}_{\mathsf{B}}}}

\DeclareMathOperator{\spl}{SPL}
\DeclareMathOperator{\Id}{I_{\mathrm{d}}}

\newcommand{\su}{\subset}

\newcommand{\be}{\beta}

\renewcommand{\phi}{\varphi}

\title[Squares and their centers]
{Squares and their centers}

\author{Tam\'as Keleti}
\address{Institute of Mathematics\\
E\"otv\"os Lor\'and University\\
P\'azm\'any P\'eter s. 1/c, 1117
Budapest, Hungary}
\email{tamas.keleti@gmail.com}
\urladdr{http://www.cs.elte.hu/analysis/keleti}
\thanks{Part of this research was done when the first author was a visitor
at the Alfr\'ed R\'enyi Institute of Mathematics. He was also supported by
Hungarian Scientific Foundation grant no.~104178.
}

\author{D\'aniel T. Nagy}
\address{Institute of Mathematics\\
E\"otv\"os Lor\'and University\\
P\'azm\'any P\'eter s. 1/c, 1117
Budapest, Hungary}
\email{dani.t.nagy@gmail.com}

\author{Pablo Shmerkin}
\address{Department of Mathematics and Statistics\\
Torcuato Di Tella University and CONICET\\
Av. Figueroa Alcorta 7350\\
Buenos Aires, Argentina}
\email{pshmerkin@utdt.edu}
\urladdr{http://www.utdt.edu/profesores/pshmerkin}
\thanks{The third author was partially supported by Project PICT 2011-0436 (ANPCyT). Part of this research was done while the third author was visiting E\"{o}tv\"{o}s Lor\'{a}nd University.}

\begin{document}
\begin{abstract}
We study the relationship between the sizes of two sets $B, S\subset\mathbb{R}^2$ when $B$ contains either the whole boundary, or the four vertices, of a square with  axes-parallel sides and center in every point of $S$, where size refers to one of cardinality, Hausdorff dimension, packing dimension, or upper or lower box dimension. Perhaps surprinsingly, the results vary depending on the notion of size under consideration. For example, we construct a compact set $B$ of Hausdorff dimension $1$ which contains the boundary of an axes-parallel square with center in every point $[0,1]^2$, but prove that such a $B$ must have packing and lower box dimension at least $\tfrac{7}{4}$, and show by example that this is sharp. For more general sets of centers, the answers for packing and box counting dimensions also differ. These problems are inspired by the analogous problems for circles that were investigated by Bourgain, Marstrand and Wolff, among others.
\end{abstract}

\keywords{Squares, square vertices, Hausdorff dimension, box dimension, packing dimension}

\subjclass[2010]{Primary: 05B30, 28A78; Secondary: 05D99, 11P99, 42B25, 52C30}

\maketitle

\section{Introduction and statement of results}

\subsection{Introduction}

The inspiration for this work arose from a beautiful and deep result due independently to Bourgain \cite{Bourgain86} and Marstrand \cite{Marstrand87}: if a set $B\subset\R^2$ contains a circle with center in every point of the plane, then $B$ has positive Lebesgue measure. This result has been sharpened in many ways. Bourgain himself proved stronger, and sharp, $L^p$ bounds for the associated circle maximal operator, from which it follows as a corollary; his result extends to other curves with non-zero curvature everywhere. Wolff \cite[Corollary 3]{Wolff00} proved, as a corollary of strong smoothing estimates, that if $S$ is a subset of the plane with $\hdim S>1$ (where $\hdim$ stands for Hausdorff dimension) and $B$ contains a circle with center in each point of $S$, then the conclusion that $B$ has positive measure continues to hold. A real line variant of Bourgain's maximal operator bounds, with circles replaced by suitable Cantor sets, was recently established by {\L}aba and Pramanik \cite{LabaPramanik11}.

The purpose of this paper is to study variants of this kind of problems, in which circles are replaced by squares. Here, and throughout the paper, by a square we mean the boundary of a square with axes-parallel sides, unless otherwise indicated. The solutions to the circle problems described above involve a fair amount of intricate geometry related to the way in which families of thin annuli can intersect (in most cases, in addition to Fourier analytic techniques). By contrast, (neighborhoods of) squares intersect in a much simpler fashion: in essence, one only needs to understand the family of lines containing the sides, which is a one-dimensional problem. Also, families of squares with far away centers can have very large intersection (they may share part of a side), which is not the case for circles. Hence, although our problem is in some sense geometrically simpler, the answers are strikingly different and, as we will see, new phenomena emerge.

The problems we study, and the circle analogs that inspired them, belong to a wider family of ``Kakeya type'' problems, of which there are many important examples in geometric measure theory and harmonic analysis. They all share the following structure: for each $x$ in some parameter space $S$, there is a family of sets $\mathcal{F}_x\subset\mathbb{R}^d$, and one would like to understand how ``small'' can a set $B\subset\R^d$ be, given  that it contains an element of $\mathcal{F}_x$ for all $x\in S$. Here ``small'' might refer to Lebesgue measure, or some fractal dimension. Some well known examples include the Kakeya and Furstenberg problems (see e.g. \cite{Wolff99}), as well as a different problem involving circles of Wolff \cite{Wolff97}, in which $S=(0,\infty)$ and $\mathcal{F}_r$ is the family of circles with radius $r$.

\subsection{Statement of main results}

It was likely known to Wolff and others that a set containing a square centered at every point of the unit square can have zero Lebesgue measure (unlike the situation for circles). Our first main result says that, perhaps surprisingly, the Hausdorff dimension of such a set can be $1$ - the same as that of a single square, even if the set is required to be closed.

\begin{theorem} \label{t:Hausdorff}
There exists a closed set $B\subset\R^2$ of Hausdorff
dimension $1$ containing the boundary of a square with axes-parallel sides with center in every point of $\R^2$.
\end{theorem}

Note that if the set of centers is some nonempty subset $S\subset\R^2$, then the smallest possible dimension of $B$ is still $1$.  It will be clear from the proof that if we wanted centers only in some bounded subset of $\R^2$, then $B$ could be taken to be compact.

This counter-intuitive result suggests that perhaps Hausdorff dimension is not the ``correct'' notion of size for this problem. Other useful notions of dimension are packing dimension $\pdim$ and lower and upper box counting dimensions, $\lbdim$ and $\ubdim$. See e.g. \cite[Chapter 3]{Falconer03} for their definitions. For these notions of dimension, we obtain a different answer.

\begin{theorem} \label{t:packingbox-all-centers}
Let $\dim$ be one of $\pdim, \lbdim$ or $\ubdim$. Then $\dim B\ge \tfrac{7}{4}$ for any set $B$ which contains the boundary of a square with axes-parallel sides with center in every point of $[0,1]^2$. Moreover, there exists a compact set $B$ with this property, such that $\pdim B=\bdim B=\tfrac{7}{4}$
\end{theorem}
Here, and in the sequel, whenever we take the (upper or lower) box counting dimension of a set, this is assumed to be bounded. To the best of our knowledge, this is the first instance in which the critical dimension for a Kakeya-type problem is known to depend on the notion of dimension under consideration.

Unlike the Hausdorff dimension problem, now it makes sense to restrict the centers to a subset $S\su\mathbb{R}^2$, and the natural set of problems to consider is the relationship between $\dim B$ and $\dim S$ for a given notion of dimension $\dim$ (i.e. we impose the same notion of ``size'' for the set of centers and the union of squares). Perhaps surprisingly, we now get a different answer for packing dimension and for box counting dimensions.

\begin{theorem} \label{t:packingbox-arbitrary-centers}
Let $S, B\subset\R^2$ be sets such that $B$ contains the boundary of a square with axes-parallel sides and center in every point of $S$. Then:
\begin{enumerate}
\item If $\dim=\ubdim$ or $\lbdim$, then $\dim B\ge \max(1,\tfrac{7}{8}\dim S)$.
\item $\pdim B\ge 1+\tfrac{3}{8}\pdim S$.
\end{enumerate}

Conversely, for each $s\in [0,2]$ there are compact sets $S, B$ as above such that:
\begin{enumerate}
\item[(a')] $\bdim S=s$ and $\bdim B= \max(1,\tfrac{7}{8}s)$.
\item[(b')] $\pdim S=s$ and $\pdim B=1+\tfrac{3}{8}s$.
\end{enumerate}
\end{theorem}

A related but more combinatorial problem concerns replacing the whole boundary of the square by the four vertices. We will see that, in some cases, both problems turn out to be closely related. Note that for the vertices problem, we no longer have the trivial automatic lower bound $1$, so even for Hausdorff dimension the answer a priori could, and indeed does, depend on the size of the set of centers.

\begin{theorem} \label{t:vertices}
Let $S, B\subset\R^2$ be sets such that $B$ contains the four vertices of a square with axes-parallel sides and center in every point of $S$. Then:
\begin{enumerate}
\item $\hdim B\ge \max(\hdim S-1,0)$. In particular, if $S=[0,1]^2$, then $\hdim B\ge 1$.
\item If $\dim$ is one of $\pdim$, $\ubdim$ or $\lbdim$, then $\dim B\ge \tfrac{3}{4}\dim S$. In particular, if $S=[0,1]^2$, then $\dim B\ge \tfrac{3}{2}$.
\end{enumerate}
Conversely, for each $s\in [0,2]$ there are compact sets $S, B$ as above such that:
\begin{enumerate}
\item[(a')] $\hdim S=s$ and $\hdim B=\max(s-1,0)$.
\item[(b')] $\pdim S=\bdim S=s$ and $\pdim B=\bdim B= \tfrac{3}{4}s$.
\end{enumerate}
In the case $s=2$, we can take $S=[0,1]^2$.
\end{theorem}

For most of these problems, it makes sense to consider also squares with sides pointing in arbitrary directions.
Altough we do not know much in this setting, we have the following proposition that (together with the first part of Theorem~\ref{t:vertices}) shows that the answer can be different if we allow this additional degree of freedom.

\begin{prop} \label{p:2D-hausdorff-rotated}
There exists a closed set $B\su\R^2$ of Hausdorff dimension zero that contains the
vertices of at least one (possibly rotated) square around each point of $\R^2$.
\end{prop}

There are natural discrete versions of the above problems (in which the sets are finite and dimension is replaced by cardinality). We start with the version for vertices; this is the most combinatorial of the results in this paper (it has an additive combinatoric flavor, although we do not know of any connection with established results in this area). Furthermore, it is key to many of the proofs of the estimates for packing and box counting dimensions stated above, not just for problems involving vertices but also when the whole square boundaries are considered.

\begin{theorem} \label{t:discrete-vertices}
\begin{enumerate}
\item Let $B\subset\mathbb{R}^2$ be a finite set, and let
\[
S=\{(x,y)\in\mathbb{R}^2~|~\exists r~(x-r,~y-r),~(x+r,~y-r),~ (x-r,~y+r),~(x+r,y+r)\in B\}.
\]
Then $|S|\leq (2 |B|)^{\frac{4}{3}}$.
\item For some universal constant $C>0$ one can find sets $B, S$ as above, with $|S|$ of arbitrary cardinality, such that $|S|\geq C\,|B|^{\frac{4}{3}}$.
\end{enumerate}
\end{theorem}

In the discrete problem for square boundaries, we consider finite subsets of $\Z^2$, and the intersections of the square boundaries with $\Z^2$.

\begin{theorem} \label{t:discrete-squares}
\begin{enumerate}
\item Let $B\subset \mathbb{Z}^2$ a finite set, and let $S$ be the set of points that are centers of (discrete) square boundaries contained in $B$. Then  $|B|\geq \Omega\left((|S|/\log|S|)^{\frac78}\right)$.

\item Conversely, there exist $S, B$ as above, with $|S|$ of arbitrary cardinality, such that $|B|\leq O(|S|^{\frac{7}{8}})$.
\end{enumerate}
\end{theorem}
We believe it should be possible to eliminate the $\log |S|$ term in the above theorem.

We summarize our results in the following table. In all cases $s$ is the size of the set of centers $S$, and the table gives the smallest possible size of  sets $B$ that contain the vertices of a square/the boundary of a square with centers in every point of $S$. In the case of cardinality, these numbers should be understood as the logarithm of the cardinality of the corresponding sets, up to smaller order factors.

\medskip

\begin{center}
\begin{tabular}{|c|c|c|}
 \hline
 Notion of size & vertices problem & boundaries problem \\
 \hline
 $\hdim$ & $\max(s-1,0)$ & $1$ \\
 $\pdim$ & $\frac{3}{4}s$ & $1+\frac{3}{8}s$ \\
 $\ubdim$ & $\frac{3}{4}s$ & $\max(1, \frac{7}{8}s)$ \\
 $\lbdim$ & $\frac{3}{4}s$ & $\max(1, \frac{7}{8}s)$ \\
 $|\cdot|$ & $\frac{3}{4}s$ & $\frac{7}{8}s$\\
 \hline
\end{tabular}
\end{center}

\subsection{Structure of the paper}

The rest of the paper is organized as follows. In Section \ref{sec:preliminaries} we recall  properties of dimension that will be required in the sequel. The main results concerning Hausdorff dimension are proved in Section \ref{sec:Hausdorff}. The discrete results, Theorems \ref{t:discrete-vertices} and \ref{t:discrete-squares}, are proved in Section \ref{sec:discrete}. The lower bounds for the packing and box counting dimensions of the set $B$ are proved in Section \ref{sec:estimates}, while the examples showing their sharpness are constructed in Section \ref{sec:constructions}. We conclude with some remarks on possible directions of future research in Section \ref{sec:remarks}.

For the most part, sections can be read independently, except for Section \ref{sec:estimates} which depends on Section \ref{sec:discrete}. In particular, the Hausdorff dimension results in Section \ref{sec:Hausdorff} are independent from the rest of the paper.

\section{Preliminaries} \label{sec:preliminaries}

We recall some basic facts on dimension; in the later sections we will call upon these without further reference.
For any set $A\subset\R^d$ one has the chains of inequalities
\begin{align*}
\hdim A & \le  \lbdim A  \le \ubdim A,\\
\hdim A & \le \pdim A  \le \ubdim A.
\end{align*}
The lower box counting and packing dimensions are not comparable.

We will often have to deal with dimensions of product sets. The following inequalities hold but can be strict:
\begin{align*}
\hdim(A)+\hdim(B) &\le \hdim(A\times B) \le \hdim(A)+\pdim(B),\\
\pdim(A\times B) &\le \pdim(A)+\pdim(B),\\
\ubdim(A\times B) &\le \ubdim(A)+\ubdim(B).
\end{align*}
See e.g. \cite[Theorem 8.10]{Mattila95} and \cite[Product Formula 7.5]{Falconer03} for the proofs.

\section{Hausdorff dimension results} \label{sec:Hausdorff}

\subsection{Small sets with large intersection, and the proofs of Theorem~ \ref{t:Hausdorff} and  Proposition \ref{p:2D-hausdorff-rotated}}

In \cite{DMT60}, Davies, Marstrand and Taylor construct a closed set $A\subset\R$ of Hausdorff dimension zero with the property that for any finite family $(f_i)_{i=1}^m$ of invertible affine maps on $\R$, the intersection $\cap_{i=1}^m f_i(A)$ is nonempty. We need to adapt their construction to suit our needs; in particular, we will make use of the analog result in the plane, and of a version in which the maps $f_i$ are taken from a fixed compact set. In the proof of Theorem \ref{t:vertices}(a') we will also need a more flexible variant of the construction.

\begin{prop} \label{p:intersecting-zero-dim-set}
For any dimension $d\ge 1$ there exists a closed set $A\subset\R^d$ of zero Hausdorff dimension, such that for any finite family of invertible affine maps $(f_i)_{i=1}^m$ on $\R^d$, the intersection $\cap_{i=1}^m f_i(A)$ is nonempty.

If the maps $f_i$ are constrained to lie in a fixed compact set of invertible affine maps, then $A$ can be taken to be compact.
\end{prop}

The proofs of Theorem \ref{t:Hausdorff} and Proposition \ref{p:2D-hausdorff-rotated} assuming this proposition are rather short. In order not to interrupt the flow of ideas, and because the proof of Proposition \ref{p:intersecting-zero-dim-set} will be needed later in the proof of Proposition \ref{p:existenceACD}, we defer its proof to Section \ref{subsec:proof-DMT}.

\begin{lemma}\label{l:easy}
For any $A\su\R$ the following two statements are equivalent.

(i)  For any $x,y\in\R$ there exists $r>0$ such that $x-r,x+r,y-r,y+r\in A$.

(ii)  For any $x,y\in\R$, $A_{x,y}\not\subset\{0\}$, where
$A_{x,y}=(A-x)\cap (x-A) \cap (A-y) \cap (y-A)$.

In particular, any set obtained from Lemma~\ref{p:intersecting-zero-dim-set}
satisfies (i) and (ii).
\end{lemma}

\begin{proof}
The equivalence of (i) and (ii) is clear. Let $A$ be a set obtained from Lemma~\ref{p:intersecting-zero-dim-set}.
Then $A_{x,y}\cap (1-A_{x,y})$ is nonempty, so $A_{x,y}$ contains a positive element, and we are done.
\end{proof}

\begin{prop}\label{p:1D-hausdorff}
There exists a closed set $A\su\R$ of zero Hausdorff dimension, such that for any $x,y\in\R$ there exists $r>0$ such that $x-r,x+r,y-r,y+r\in A$.
\end{prop}

\bp
This immediately follows from Proposition~\ref{p:intersecting-zero-dim-set}
and Lemma~\ref{l:easy}.
\ep



We can now easily deduce the proofs of Theorem~ \ref{t:Hausdorff} and  Proposition \ref{p:2D-hausdorff-rotated}.

\bp[Proof of Theorem~\ref{t:Hausdorff}]
Let $A$ be the set from Proposition~\ref{p:1D-hausdorff} and let $B=(A\times\R)\cup(\R\times A)$.
\ep

\bp[Proof of Proposition \ref{p:2D-hausdorff-rotated}]
Let $B\subset\R^2$ be the set given by Proposition \ref{p:intersecting-zero-dim-set}. This set is good since, if we rotate $B$ around any point by
$0,90,180$ and $270$ degrees, then the intersection of these four sets is nonempty and cannot be a singleton, as otherwise a further intersection with a translation of $B$ would be empty.%
\ep

\begin{remark}
If we did not insist that the sets $A, B$ from Propositions \ref{p:1D-hausdorff} and \ref{p:2D-hausdorff-rotated} be closed, we could just take $A, B$ to be dense $G_\delta$ subsets of $\R,\R^2$ of Hausdorff dimension zero. This would also give a simpler construction for Theorem~\ref{t:Hausdorff}.
\end{remark}

\subsection{Proof of Proposition \ref{p:intersecting-zero-dim-set} }
\label{subsec:proof-DMT}

\bp[Proof of Proposition \ref{p:intersecting-zero-dim-set} ]
Let $(\delta_i)_{i=0}^\infty$, $(\eps_i)_{i=1}^\infty$ be sequences of real numbers decreasing to zero with $\delta_0=1$, and such that
\begin{equation} \label{e:eps-delta}
\delta_i\leq \eps_i \leq \frac{\delta_{i-1}}{2\sqrt{d}+2}\quad\text{for all }i\in\N,
\end{equation}
and $\frac{\log \eps_i}{\log\delta_i} \to 0$. Let $(t_i)_{i\in\N}$ be arbitrary points in $\R^d$. For $i,p\in\N$, let
\begin{align}
F_i &= \bigcup_{k\in\Z^d} B(t_i+\eps_i k, \delta_i),\label{e:Fi}\\
K_p &= \bigcap_{j=1}^\infty F_{(2j-1)2^{p-1}}.\label{e:Kp}
\end{align}
(For concreteness we can take $t_k=0$ for all $k$, but the additional flexibility will be required later.)

Denote the identity map of $\R^d$ by $\Id$, and let $\mathcal{U}$ be the set of affine maps $f(x)=Sx+t$ on $\R^d$, such that all singular values of $S$ are strictly larger than $1$, $\|S-\Id\|<1/3$, and $\|t\|<1$. Note that this set is open and nonempty, and hence meets any dense set of affine maps.

Fix any dense set $(g_j)_{j=1}^\infty$  of invertible affine maps. For each $M\in\N$, let $B_M$, $B'_M$ be balls such that
\beq \label{e:union-disjoint-from-ball}
B(0,M) \cap \bigcup_{p=1}^M g_p(B_M) = \varnothing,
\eeq
$B'_M$ has radius $\ge 1$, and $B'_M\subset h(B_M)$ for all $h\in\mathcal{U}$. For example, we could take $B_M=B(v, \|v\|/2)$ and $B'_M=B(v,1)$ for a vector $v$ of sufficiently large norm.

For the first part, the desired set is
\[
A = \bigcup_{M=1}^\infty \bigcup_{p=1}^M g_p(K_p \cap B_M).
\]
The set $A$ is indeed closed, since $A\cap B(0,M)$ is a finite union of closed sets by \eqref{e:union-disjoint-from-ball}. Let us next see that $\hdim A=0$. Since $A \subset \cup_{p=1}^\infty g_p(K_p)$, it is enough to check that $\hdim K_p\cap Q=0$ for all $p$ and all balls $Q$ of unit radius. But $K_p\cap Q$ can be covered by $O(\eps_i^{-d})$ balls of radius $\delta_i$ so, since $\log \eps_i/\log \delta_i\to 0$, we see that indeed  $\hdim K_p\cap Q=0$.

It remains to show that if $(f_i)_{i=1}^m$ are invertible affine maps, then the intersection $\cap_{i=1}^m f_i(A)$ is nonempty.  For each $i=1,\ldots,m$, we can pick $p(i)$ such that $h_i:=f_i \circ g_{p(i)}\in\mathcal{U}$. Set $M=\max_{i=1}^m p(i)$. Then
\begin{align*}
\bigcap_{i=1}^m f_i(A) &\supset \bigcap_{i=1}^m  \bigcup_{p=1}^M f_i \circ g_p(K_p\cap B_M)\\
&\supset \bigcap_{i=1}^m h_i\left(K_{p(i)}\cap B_M\right)\\
&\supset \bigcap_{i=1}^m h_i\left(K_{p(i)}\right) \cap \bigcap_{i=1}^m h_i(B_M)\\
&\supset \bigcap_{i=1}^m \bigcap_{j=1}^\infty h_i\left(F_{(2j-1) 2^{p(i)-1}}\right) \cap B'_M.
\end{align*}

We claim that if $h'_j\in\mathcal{U}$ and $Q'$ is a closed ball of radius $1$, then $\bigcap_{j=1}^q h'_j(F_j)\cap Q'$ contains a closed ball of radius $\delta_q$. We prove this by induction. The case $q=0$ is trivial. Let $x$ be the center of a ball of radius $\delta_q$ contained in $\bigcap_{j=1}^q h'_j(F_j)\cap Q'$. Since $h'_{q+1}\in\mathcal{U}$, the image $h'_{q+1}(F_{q+1})$ contains the union of balls with of radius $\delta_{q+1}$ with centers in $h'_{q+1}(t_{q+1}+\eps_{q+1}\Z^d)$, which is a $(2\sqrt{d}\eps_{q+1})$-dense set (since $\|h'_{q+1}\|<2$ and $\Z^d$ is $\sqrt{d}$-dense). Hence there exists $y\in h'_{q+1}(t_{q+1}+\eps_{q+1}\Z^d)$ with $|y-x|<2\sqrt{d}\eps_{q+1}$. In light of \eqref{e:eps-delta}, $B(y,\delta_{q+1}) \subset B(x,\delta_q)$, and we obtain the claim.

The proof of the first part is finished by applying the claim to $h'_k=h_i$ if $k=(2j-1) 2^{p(i)-1}$ for some $i,j$, and $h'_k=\Id$ otherwise.

Now consider the case in which the $f_i$ are taken from some compact set $\mathcal{C}$. Note that if $f \circ g_i\in\mathcal{U}$, then $\widetilde{f} \circ g_i\in\mathcal{U}$ for $\widetilde{f}$ in a neighborhood of $f$. Hence it follows from the compactness of $\mathcal{C}$ that there is a number $M=M(\mathcal{C})$ such that that for any $f_1,\ldots,f_m\in\mathcal{C}$ we can find $p(1),\ldots, p(m)\in \{1,\ldots,M\}$ such that $f_i \circ g_{p(i)}\in\mathcal{U}$. Set
\[
A' = \bigcup_{p=1}^M g_p(K_p\cap B_M).
\]
This set is clearly compact, and the previous arguments show that $\hdim A'=0$ and $\cap_{i=1}^m f_i(A')\neq\varnothing$ when $f_i\in\mathcal{C}$.
\end{proof}

\subsection{Hausdorff dimension and vertices}

We now prove the part of Theorem~\ref{t:vertices} concerning Hausdorff dimension. For simplicity we restate it here, in a slightly stronger form (for the lower bound on $\hdim B$, it is enough to assume that \emph{some} vertex of the square is in $B$).

\begin{theorem} \label{t:vertices-Hdim}
 Suppose $S, B\subset\R^2$ are such that for each $s\in S$, $B$ contains at least one vertex of a square with center in $s$. Then $\hdim B\ge \max(\hdim S-1,0)$.

 Moreover, this is sharp in a strong way: for every $s\in [1,2]$, there are compact sets $S, B\subset\R^2$ such that $\hdim S=s, \hdim B=s-1$, and $B$ contains \emph{all} vertices of a square with center in all points of $S$. If $s=2$, $S$ can be taken to be $[0,1]^2$.
\end{theorem}

The first part of the theorem is easy, and the main difficulty is to construct an example showing it is sharp in the sense of the second part. For this the key is to construct sets $A, C, D$ satisfying the properties given in the following proposition.

\begin{prop} \label{p:existenceACD}
 For every $s\in [1,2]$, there exist compact sets $A,C,D\subset\R$ such that:
 \begin{enumerate}[(i)]
  \item For every $x,y\in [0,1]$, there is $r>0$ such that $x-r,x+r,y-r,y+r\in A$.
  \item $\hdim A\times C=\hdim A\times D=s-1$.
  \item $\hdim(C\times D)=s$.
 \end{enumerate}
 Moreover, if $s=2$, then we can take $C=D=[0,1]$.
\end{prop}

We first show how to deduce Theorem \ref{t:vertices-Hdim} from this proposition, and give the proof of the proposition in the remainder of the section.

\begin{proof}[Proof of Theorem \ref{t:vertices-Hdim}]
It is more convenient to work with 45 degree rotations; hence we use the convention that if $F\subset\R^2$, then $F'$ is its $45$ degree rotation.

For the first part, decompose $S=\cup_{i=1}^4 S_i$ where $B$ contains an upper-left vertex of a square with center in every point of $S_1$, and likewise for $S_2, S_3, S_4$ and the remaining positions of the vertices. Then $\hdim S=\hdim S_i$ for some $i$; without loss of generality, $i=1$ and we may assume $B$ contains an upper-left vertex of a square with center in every point of $S$. Now we only need to notice that if $\pi(x,y)=x$ is the projection onto the $x$-axis, then $\pi(S')\subset\pi(B')$, and therefore $S'\subset  B'\times\R$, so $\hdim S\le 1+\hdim B$ and we are done.

For the second part, we take $S,B$ such that $S'=C\times D$, $B'=(A\times C)\cup (D\times A)$, where $A, C, D$ are the sets from Proposition \ref{p:existenceACD}. Then $\hdim S=s, \hdim B=s-1$. Moreover, for every $x,y\in [0,1]$ (in particular, for every $x\in C, y\in D$), there is $r>0$ such that $x-r,x+r,y-r,y+r\in A$, and this implies that $B$ contains the vertices of a square with center in every point of $S$, as desired. If $s=2$, then $S'=[0,1]^2$; as the problem is invariant under homotheties, there is an example in which $S$ contains $[0,1]^2$, as claimed.
\end{proof}

For the proof of Proposition \ref{p:existenceACD} we will need a standard construction which consists in pasting together a countable sequence of sets  along dyadic scales. In order to define this operation it is more convenient to use symbolic notation. We work in an ambient dimension $d$ (in our later application, $d$ will be either $1$ or $2$). Given a finite sequence $\ii=(i_1,\ldots,i_n)\in \Lambda^n$, where $\Lambda=\{0,1\}^d$, we define
\[
Q(\ii) = \prod_{k=1}^d \left[\sum_{j=1}^n 2^{-j} (i_j)_k, \sum_{j=1}^n 2^{-j} (i_j)_k+ 2^{-n}\right].
\]
In other words, $Q(\ii)$ is the closed dyadic cube of side-length $2^{-n}$ whose position is described by the sequence $\ii$. We can also express $Q(\ii)$ as the set of those points $(x_1,\ldots,x_d)\in[0,1]^d$ for which every $x_j$ can be written in base $2$ so that the first $n$ digits after the binary point are  $(i_1)_j,\ldots, (i_n)_j$.

Given a set $X\subset [0,1]^d$ and $n\in\N$ let
\[
\mathcal{Q}(X,n)=\{ \ii\in \Lambda^n: Q(\ii)\cap X\neq\varnothing\},
\]
that is, $\mathcal{Q}(X,n)$ consists of sequences describing the cubes of step $n$ that hit $Q$.

We can now describe the splicing operation for sets. We fix a strictly increasing sequence $(a_n)_{n=0}^\infty$ of natural numbers with $a_0=0$, which is rapidly increasing in the sense that $a_n/a_{n+1}\to 0$. For example, we could take $a_n=2^{2^n}-1$. Now given a sequence $X=(X_i)_{i=1}^\infty$ of subsets of $[0,1]^d$, we define the splicing $\spl(X)= \bigcap_{n=1}^\infty E_n$, where
\begin{equation} \label{e:defEn}
E_n = \bigcup\left\{ Q(\ii_1\ii_2\ldots \ii_n): \ii_j \in \mathcal{Q}(X_j,a_j-a_{j-1})\right\}.
\end{equation}
(Here $\ii_1\ii_2\ldots \ii_n$ is obtained by concatenating the corresponding sequences.) Splicing preserves cartesian products: $\spl(X\times X')=\spl(X)\times \spl(X')$, where $(X\times X')_n =X_n\times X'_n$. This property will be exploited later.

The next lemma gives the value of the Hausdorff dimension of $\spl(X)$ under some assumptions.
\begin{lemma} \label{l:Hdimcalculation}
 Let $\mathcal{Z}=\{Z_1,\ldots, Z_m\}$ be a finite collection of subsets of $[0,1]^d$. Let $X=(X_j)_{j=1}^\infty$ be a sequence of subsets of $[0,1]^d$ such that $X_j\in\mathcal{Z}$ for all $j$, and each $Z_i$ appears infinitely often in $X$.

 Then
 \begin{align*}
 \lbdim \spl(X) &\le \min(\ubdim Z_1,\ldots,\ubdim Z_m),\\
 \hdim \spl(X) &\ge \min(\hdim Z_1,\ldots,\hdim Z_m).
 \end{align*}
 In particular, if $\dim_H Z_i=\dim_B Z_i$ for all $i$, then
 \[
 \hdim \spl(X) = \min(\hdim Z_1,\ldots,\hdim Z_m).
 \]
\end{lemma}
\bp
Write $E=\spl(X)$, and $\pi:\N\to\{1,\ldots,m\}$ for the map $X_j = Z_{\pi(j)}$. We first prove the upper bound. Let $i$ be such that $\ubdim Z_i$ is minimal, and pick any $j\in\pi^{-1}(i)$. Then
\[
|\mathcal{Q}(E,a_j)| \le 2^{a_{j-1}} |\mathcal{Q}(Z_i,a_j-a_{j-1})|.
\]
Hence
\begin{align*}
\lbdim(E) &\le \liminf_{j\to\infty,j\in\pi^{-1}(i)} \frac{\log_2 |\mathcal{Q}(E,a_j)| }{a_j}\\
&\le  \liminf_{j\to\infty, j\in\pi^{-1}(i)}  \frac{a_{j-1}+\log_2 |\mathcal{Q}(Z_i,a_j-a_{j-1})|}{a_j}\\
&= \liminf_{j\to\infty, j\in\pi^{-1}(i)}  \frac{\log_2 |\mathcal{Q}(Z_i,a_j-a_{j-1})|}{a_j-a_{j-1}} \le \ubdim(Z_i).
\end{align*}

We now prove the lower bound. If $\dim_H(Z_i)=0$ for some $i$, there is nothing to do. Otherwise, let $0<s<\hdim(Z_i)$ for all $i$. By Frostman's Lemma (see e.g. \cite[Corollary 4.12]{Falconer03}), there are a constant $C$ and measures $\mu_i$ supported on $Z_i$ ($i=1,\ldots,m$) such that $\mu_i(B(x,r))\le C\, r^s$ for all $x\in [0,1]^d, r>0$.

We construct a measure $\nu$ supported on $E$, in a similar way to the construction of $E$. Namely, suppose $\ii\in \Lambda^n$, where $a_k \le n < a_{k+1}$. Decompose $\ii=(\ii_1,\ldots,\ii_{k+1})$, where $\ii_j\in \Lambda^{a_j-a_{j-1}}$ if $j=1,\ldots,k$, and $\ii_{k+1}\in \Lambda^{n-a_k}$. Then we define
\[
\nu(Q(\ii)) = \mu_1(Q(\ii_1))\cdots \mu_{k+1}(Q(\ii_{k+1})).
\]
It is easy to check that this does define a Borel measure on $[0,1]^d$, which is in fact supported on $E$. Now by the Frostman condition,
\[
\nu(Q(\ii)) \le C^{k+1} 2^{-sn} = O(2^{(\eps-s)n})
\]
for any $\eps>0$. By the mass distribution principle (see \cite[Mass Distribution Principle 4.2]{Falconer03}), $\hdim E\ge s-\eps$, so after letting $s\to \min_i\hdim(Z_i),\eps\to 0$ we are done.
\ep

\begin{proof}[Proof of Proposition \ref{p:existenceACD}]

Let $(a_j)$ be the sequence introduced previously. 
Let $t_j=\de_j=2^{-a_{2j}}/2$ and $\eps_j=2^{-a_{2j-1}}$,
and note that $\frac{\log\eps_j}{\log\delta_j}\to 0$ and $\delta_j\ll\eps_j\ll\delta_{j-1}$.

Let $\mathcal{C}$ be the family of affine maps of the form $\pm x+ b$ with $b\in [-2,2]$
 and let $A$ be the compact set constructed in the proof of Proposition \ref{p:intersecting-zero-dim-set} using these sequences and for affine maps taken from $\mathcal{C}$. Then Lemma~\ref{l:easy} shows that \textit{(i)} holds.

Let $F_j, K_p$ be as in Equations \eqref{e:Fi},\eqref{e:Kp} in the proof Proposition \ref{p:intersecting-zero-dim-set}. With our choice of sequences, we get
\[
F_j = \bigcup_{k\in\Z} [k 2^{-a_{2j-1}},k 2^{-a_{2j-1}} + 2^{-a_{2j}}].
\]
Hence $F_j$ contains exactly those real numbers that can be written in base $2$ so
that for every $a_{2j-1}<i\le a_{2j}$ the $i$-th digit after the binary point is $0$.
This allows us to express the sets $F_j\cap [0,1]$ and $K_p\cap [0,1]$ in the language of splicing.
Note that $F_j\cap[0,1]=\spl(X)$ if $X_{2j}=\{0\}$ and $X_i=[0,1]$ for $i\neq 2j$.
Thus, if we define a sequence $X^{(p)}$ by $X^{(p)}_n=\{0\}$ if $n$ is of the form $(2j-1)\cdot 2^p$ with $j,p\in\N$, and $X^{(p)}_n=[0,1]$ otherwise,
we get
\[
K_p\cap[0,1] = \bigcap_{j=1}^\infty \left(F_{(2j-1)2^{p-1}}\cap [0,1]\right) =  \spl(X^{(p)}).
\]
We will now define the desired sets $C, D$. Let $B\subset [0,1]$ be any set with $\hdim B=\bdim B=s-1$. For $s=2$, take $B=[0,1]$. Define sequences $X'_n, X''_n$ of subsets of $[0,1]$ as follows:
\begin{align*}
X'_{n} = B \text{ if $n=(4j-3)2^p$ for some $j,p\in\N$, and } X'_n = [0,1] \text{ otherwise},\\
X''_{n} = B \text{ if $n=(4j-1)2^p$ for some $j,p\in\N$, and } X'_n = [0,1] \text{ otherwise},
\end{align*}
and set $C=\spl(X'), D=\spl(X'')$. Note that if $s=2$, then $C=D=[0,1]$. It follows from Lemma \ref{l:Hdimcalculation} that $\hdim(C)=\hdim(D)=\hdim(B)=s-1$. Moreover,  since $C\times D= \spl(X'\times X'')$, and $(X'\times X'')_n$ is one of $B\times [0,1]$, $[0,1]\times B$, $[0,1]^2$, with each of these appearing infinitely often, we get
$\hdim(C\times D)=\hdim(B\times [0,1])=s$, using Lemma \ref{l:Hdimcalculation} again.

It remains to show that $\hdim(A\times C)=\hdim(A\times D)=s-1$. We prove this for $C$; for $D$ the argument is the same.

Clearly $\hdim(A\times C)\ge \hdim(C)=s-1$, so we need to establish the upper bound. Note that with our choices the $F_j$ are $1$-periodic, and therefore so is $K_p$. Recall from the proof of Proposition \ref{p:intersecting-zero-dim-set} that $A= \bigcup_{p=1}^M g_p(K_p)$, where $g_1,\ldots,g_M$ are affine maps. Hence it is enough to show that $\hdim(C\times (K_p\cap [0,1]))\le s-1$ for each $p$. But $C\times (K_p\cap [0,1])= \spl(X'\times X^{(p)})$, and each $X'_n\times X^{(p)}_n$ is either $B\times \{0\}$, $B\times [0,1]$, $[0,1]\times \{0\}$ or $[0,1]^2$, with each of these appearing infinitely often. All these sets have equal Hausdorff and box-counting dimension, and the smallest dimension is $s-1=\hdim(B\times \{0\})$, hence a final application of Lemma \ref{l:Hdimcalculation} gives $\hdim(C\times (K_p\cap [0,1]))= s-1$. This finishes the proof.
\end{proof}

\section{Proofs of the discrete results} \label{sec:discrete}

\subsection{General bounds} The following is the first part of Theorem~\ref{t:discrete-vertices}. We state it separately as the proofs of both parts are unrelated, and also because it will be key for the remaining estimates as well.

\begin{lemma}[{\bf Two-Dimensional Main Lemma}] \label{l:main2D}
Let $B\subset\mathbb{R}^2$ be a finite set, and let
$$S=\{(x,y)\in\mathbb{R}^2~|~\exists r~(x-r,~y-r),~(x+r,~y-r),~ (x-r,~y+r),~(x+r,y+r)\in B\}.$$

Then $|S|\leq (2 |B|)^{\frac{4}{3}}$.
\end{lemma}

\bp Assume that a line $\ell$ with gradient $\pm 1$ intersects $S$. Then for each $p\in S\cap \ell$, $\ell$ contains two points of $B$ which are equidistant from $x$. This implies that $|S\cap \ell|\leq \binom{|B\cap \ell|}{2}$, so $|B\cap \ell|\geq |S\cap \ell|^{1/2}$.

Assume that there are $k$ lines with gradient 1 intersecting $S$, and they intersect it $p_1, p_2, \dots, p_k$ times. Also assume that there are $m$ lines with gradient -1 intersecting $S$, and they intersect it $q_1, q_2, \dots, q_m$ times. Then
\[
|S|=\sum_{i=1}^k p_i=\sum_{j=1}^m q_i,
\]
\begin{align*}
|B|&\geq \sum_{i=1}^k \sqrt{p_i},\\
|B|&\geq \sum_{j=1}^m \sqrt{q_i}.
\end{align*}

Divide the numbers $p_1, \dots, p_k$ into two groups. Let $a_1, \dots, a_v$ be the ones that are smaller than $\sqrt{|S|}$, and let $b_1, \dots,b_w$ be the remaining ones. Note that $w\leq \sqrt{|S|}$.
\medskip

Case 1: $a_1+a_2+\dots +a_v\geq \frac{|S|}{2}$.

\[
|B|\geq \sum_{i=1}^k \sqrt{p_i}\geq \sum_{i=1}^v \sqrt{a_i}\geq \sum_{i=1}^v \frac{a_i}{|S|^\frac{1}{4}}
\geq |S|^{-\frac{1}{4}} \frac{|S|}{2}=\frac{1}{2}\cdot |S|^{\frac{3}{4}}.
\]
\medskip

Case 2: $b_1+b_2+\dots +b_w\geq \frac{|S|}{2}$.

Consider the lines of gradient 1 that contain at least $\sqrt{|S|}$ points of $S$. Color all points of $S$ on these lines red. So we have $b_1+b_2+\dots +b_w\geq \frac{|S|}{2}$ red points. Let $q'_j$ denote the number of red points on the $j$th line with gradient -1 (this line has $q_j$ points of $S$ in total). Then obviously
\[
q'_j\leq \min(q_j,w)\leq \min(q_j, \sqrt{|S|}),
\]
and hence
\[
|B|\geq \sum_{j=1}^m \sqrt{q_i}\geq \sum_{j=1}^m \sqrt{q'_i}\geq \sum_{j=1}^m \frac{q'_i}{|S|^{\frac{1}{4}}}
\geq |S|^{-\frac{1}{4}} \frac{|S|}{2}=\frac{1}{2}\cdot |S|^{\frac{3}{4}}.
\]
\ep

As an immediate corollary, we get:

\begin{lemma}[{\bf One-Dimensional Main Lemma}]\label{l:main1D}
Let $A\subset\mathbb{R}$ be a finite set, and let
$$S=\{(x,y)\in\mathbb{R}^2~|~\exists r~x-r,~x+r,~y-r,~y+r\in A\}.$$

Then $|S|\leq  2^{\tfrac{4}{3}} |A|^{\tfrac{8}{3}}$.
\end{lemma}

\bp Apply the above Two-Dimensional Main Lemma to $B=A\times A$.
\ep

The above lemma will be key in the proof of the first part of Theorem~\ref{t:discrete-squares}:

\bp[Proof of Theorem~\ref{t:discrete-squares}(a)]
We may assume that $B=\cup_{s\in S}\partial Q(s,r(s))$ for some function $r:S\to\N$, where $\partial Q(s,r)$ is the discrete square boundary with center $s$ and side length $r$.
For each $j\in\N$, let $S_j=\{s\in S: r(s)\in [2^{j-1},2^j)\}$ and $B_j=\cup_{s\in S_j}\partial Q(s,r(s))$. We may assume $S_j$ is empty for $j\ge \log|S|$, otherwise $|B|\ge |S|$ and we are done. Hence we can pick some $j$ such that $|S_j|\ge \Omega(|S|/\log |S|)$; we work with this $j$ for the rest of the proof.  Note that it is enough to show that $|B_j|  \ge \Omega(|S_j|^{7/8})$.

Split $\Z^2$ into disjoint squares of side length $2^j$, and let $\mathcal{R}_{j,k}$ be the collection of those squares $R$ such that $|R\cap S_j|\in [2^{k-1},2^k)$.

Suppose $R\in\mathcal{R}_{j,k}$, and write $B_R=\bigcup_{s\in R\cap S_j} \partial Q(s,r(s))$. The key of the proof is to obtain a good lower estimate for $|B_R|$, which we can do thanks to the One-Dimensional Main Lemma. Indeed, let $A'_R, A''_R$ be the set of $x,y$ coordinates of the sides of $\partial Q(s,r(s))$ for $s\in R\cap S_j$. Then the set $A_R=A'_R\cup A''_R$ has the property that for each $s=(x,y)\in R\cap S_j$ there is $r=r(s)$ such that $x-r,x+r,y-r,y+r\in A_R$. Hence the One-Dimensional Main Lemma yields that $|A_R|\ge |R\cap S_j|^{3/8}/\sqrt{2}$. On the other hand, $B_R$ contains either $|A_R|/2$ disjoint vertical segments or $|A_R|/2$ disjoint horizontal segments of length $2^{j-1}$. Therefore
\[
|B_R| \ge 2^{j-1}  \frac{|R\cap S_j|^{3/8}}{2\sqrt{2}}  \ge \Omega(1) 2^j 2^{3k/8}.
\]

We note that when $\mathcal{R}_{j,k}$ is nonempty, we have the trivial estimates $2^{k-1} \le |R\cap S_j| \le |R_j|  = 2^{2j}$, and $|S_j| \ge |R\cap S_j|
\ge 2^{k-1}$. Also,
\[
\sum_k 2^k  |\mathcal{R}_{j,k}| \ge |S_j|.
\]
Since, for fixed $j$, each point in $\Z^2$ belongs to at most $9$ discrete square boundaries with centers in different squares of the partition and side length at most $2^j$, in estimating $|B_j|$ via $\sum_R |B_R|$ we are counting each point at most $9$ times, so we can estimate
\begin{align*}
|B_j| &\ge \Omega(1)\sum_k |\mathcal{R}_{j,k}|  2^j 2^{3k/8}\\
&\ge \Omega(1)\sum_k (2^k|\mathcal{R}_{j,k}|)  2^{k/2} (2^{-k} 2^{3k/8})\\
&\ge  \Omega(1)\sum_k (2^k|\mathcal{R}_{j,k}|) 2^{-k/8}\\
&\ge \Omega(1) \sum_k (2^k|\mathcal{R}_{j,k}|) |S_j|^{-1/8} \\
&\ge \Omega(|S_j|^{7/8}).
\end{align*}

\ep

\subsection{Constructions}

Next, we show the sharpness (up to a $\log$ factor in the case of Theorem \ref{t:discrete-squares}) of the discrete estimates we have established so far. They are all based on the construction given in the following lemma.
This construction was found independently by Bertalan Bodor,  Andr\'as M\'esz\'aros, Istv\'an Tomon, and the second author at the Mikl\'os Schweitzer Mathematical Competition
in 2012, where the first and the third authors posed a problem related to the One-Dimensional Main Lemma (Lemma~\ref{l:main1D}).

\begin{lemma}\label{l:D_k}
For any $k=1,2,\ldots$ there exists a set $D_k\su\{-k^4, -k^4+1, \dots, 2k^4\}$ such that $|D_k|\leq O(k^3)$ and
\begin{equation}\label{everything}
\forall x,y\in\{0,1,\ldots,k^4-1\}\ \exists r\in\{1,\ldots, k^4\} \ :\
x-r,x+r,y-r,y+r\in D_k.
\end{equation}
\end{lemma}

\bp
Let
$$
D_k=\Big\{a+bk+ck^2+dk^3\ :\ a,b,c,d\in\{-k+1, -k+2, \ldots, 2k-2\},~ abcd=0
     \Big\}.
$$
Then clearly $D_k\su\{-k^4, -k^4+1, \dots 2k^4\}$ and
$|D_k|\leq (3k-2)^4-(3k-3)^4=O(k^3)$, so we need to prove only
(\ref{everything}). Let $0\leq x,y\leq k^4-1$. Write them as
$$x=x_0+x_1k+x_2k^2+x_3k^3$$
$$y=y_0+y_1k+y_2k^2+y_3k^3$$
where the coefficients $x_i,y_i\in \{0,1,\ldots,k-1\}$. Let
$$r=x_0-x_1k+y_2k^2-y_3k^3.$$
Then $x-r,x+r,y-r,y+r\in D_k$ holds, since all of these numbers can be expressed as $a+bk+ck^2+dk^3$ such that the coefficients are integers between $-k+1$ and $2k-2$, and at least one of them is 0.
\ep

In the rest of the section, $D_k$ is the set from the previous lemma.

\begin{remark}
This construction shows that the One-Dimensional Main Lemma is sharp (up to a constant multiple). To see this, let $A=D_k$ and $S=\{1,2,\dots, k^4-1\}^2$. Then $|A|\leq O(k^3)$ and $|S|=k^8\geq \Omega(|A|^{\frac{8}{3}})$. Interpolating between consecutive values of $k$ (using that $(k+1)^8 = O(k^8)$) we obtain sets $S$ of arbitrary cardinality.
\end{remark}

\bp[Proof of Theorem \ref{t:discrete-vertices}]
We only have to prove the sharpness (up to constant multiple) of the Two-Dimensional Main Lemma. We can take
$B=D_k\times D_k$ and $S=\{1,2,\dots, k^4-1\}^2$. Then $|B|\leq O(k^6)$ and $|S|=k^8\geq \Omega(|B|^{\frac{4}{3}})$. Again, the cardinality of $S$ is arbitrary since we can interpolate between values of $k$.
\ep

\bp[Proof of Theorem \ref{t:discrete-squares}]
We only have to prove the second part. For this, we take $B=(D_k\times \{-k^4, -k^4+1, \dots 2k^4\})\cup (\{-k^4, -k^4+1, \dots 2k^4\}\times D_k)$ and $S=\{1,2,\dots, k^4-1\}^2$. Then there is a discrete square boundary in $B$ centered at all points of $S$, and $|B|\leq O(k^7)=O(|S|^{\frac{7}{8}})$.
\ep

\section{Box and packing dimension estimates}
\label{sec:estimates}

In this section we establish the estimates concerning packing and box counting dimensions in Theorems \ref{t:packingbox-arbitrary-centers} and \ref{t:vertices}. The examples illustrating the sharpness of these estimates are given in the next section.

\bp[Proof of Theorem \ref{t:vertices}(b)]
The statement for upper and lower box dimension is a routine deduction from the
Two-Dimensional Main Lemma (Lemma~\ref{l:main2D}). For completeness, we sketch the argument. Let $B$ and $S$ be as in the statement of the theorem, and fix $k\in\N$. Given $x\in \R^2$, let $x_k$ be the center of the half-open dyadic square of size $2^{-k}$ that contains $x$, and write $S_k=\{ x_k:x\in S\}$. Without loss of generality $B$ is the union of vertices of squares with centers in $S$. Let $S_k$ be the set obtained by replacing centers $x$ by $x_k$, and side lengths $r$ by $r_k$, the closest point to $r$ of the form $2^{-k}j, j\in\Z$ (if there are two, pick the leftmost one). Note that the new vertices are at distance $O(2^{-k})$ from the old ones, so $B$ hits $\Omega(|B_k|)$ dyadic squares of size $2^{-k}$. But it follows from the Two-Dimensional Main Lemma that $|S_k|\le 2|B_k|^{\tfrac43}$, so this gives the claim for upper and lower box dimensions.

For packing dimension, we use the well known fact (see \cite[Proposition 3.8]{Falconer03}) that
packing dimension is the same as the modified box dimension; that is,
$$
\pdim(H)=\mubdim(H)=
\inf\left\{\sup_i \ubdim(H_i)\ :\ H\su \cup_{i=1}^\infty H_i \right\}
\quad (\forall H\su\R^d, d\in\N).
$$
So let $B\su \cup_{i=1}^\infty B_i$. We need to show that
$\sup_i \ubdim(B_i)\ge \frac34\pdim S$.

Let $B'_i=\cup_{j=1}^i B_j$ and let $S_i$ consist of those points of $S$ which are centers of squares with all four vertices in $B'_i$.
By the already proved upper box dimension part of this theorem, we have
$\ubdim B_i'\ge \frac34\ubdim S_i$.
Since every point of $S$ is the center of square with all vertices in $B$, $\cup_{i=i}^\infty B_i'=B$ and $B'_1\su B'_2\su\ldots$,
we get that $S=\cup_{i=1}^\infty S_i$.
Then for every $\eps>0$ there exists an $i$ such that
$\ubdim S_i > \mubdim S - \eps = \pdim S - \eps$. Thus for this $i$ we have
$\ubdim B'_i\ge \frac34\ubdim S_i > \frac34\pdim S - \frac34 \eps$.
Since $B'_i$ is a finite union of $B_j$-s and the upper box dimension is
finitely stable (\cite[Section 3.2]{Falconer03}), this implies the existence of a
$j$ such that $\ubdim B_j > \frac34\pdim S - \frac34 \eps$, which
completes the proof.
\ep

The following is a dimension analog of the One-Dimensional Main Lemma.

\begin{prop}\label{p:dimensions-1Dcenters}
If $A$ is a set in the real line, $S$ is a set in the plane and for every
$(x,y)\in S$ there exists $r$ such that $x+r, x-r, y+r, y-r \in A$
then $\dim A \ge \frac38 \dim S$, where $\dim$ is lower or upper box dimension
or packing dimension.
\end{prop}

\bp
The proof follows exactly the same argument as the proof of Theorem~\ref{t:vertices}(b), except that we appeal to the One-Dimensional Main Lemma instead.
\ep

\begin{corollary}\label{c:dimensions-1Dcenters}
If $A$ is a subset of the real line such that for any
$(x,y)\in [0,1]\times[0,1]$
there exists $r$ such that $x+r, x-r, y+r, y-r \in A$
then the lower box dimension, upper box dimension and packing dimension
of $B$ are at least $\frac34$.
\end{corollary}

\bp[Proof of Theorem \ref{t:packingbox-arbitrary-centers}(a),(b)]
Part (a) follows directly from Theorem~\ref{t:discrete-squares}(a) (as in the proof of Theorem~\ref{t:vertices}(b)), and
from the fact that a square has box dimension $1$.

To prove (b), let
$$
B'=B+\big( (\{0\}\times \Q\big) \cup \big( \Q\times \{0\})\big) =
\left(\bigcup_{r\in\Q} B+(0,r)\right) \cup \left(\bigcup_{r\in\Q} B+(r,0)\right).
$$
Then $\pdim B'=\pdim B$. Moreover, since $B'$ contains the whole lines containing the sides of the squares that make up $B$, we have
$B'=(A_1\times\R) \cup (\R\times A_2)$, where for
every $(x,y)\in S$ there exists $r$ such that $x-r,x+r\in A_1$ and
$y-r,y+r\in A_2$.

Thus Proposition~\ref{p:dimensions-1Dcenters} can be applied
to $A=A_1\cup A_2$ and $S$, so we get $\pdim(A_1\cup A_2)\ge\frac38\pdim S$.
Therefore, either $\pdim A_1 \ge \frac38\pdim S $ or $\pdim A_2 \ge \frac38\pdim S $,
hence we conclude
$$
\pdim B = \pdim B'=
\pdim((A_1\times\R) \cup (\R\times A_2))\ge 1+\frac38\pdim S.
$$
\ep

Note that taking $s=2$ we also obtain the first part of Theorem \ref{t:packingbox-all-centers}.

\section{Constructions for box and packing dimensions}\label{sec:constructions}
\subsection{Cantor type constructions: packing dimension and the vertices problem}

Our basic construction will be obtained as an infinite sum of scaled
copies of the discrete examples. Hence first we need to calculate
the dimensions of these type of sets. This is standard, but we provide the proof for completeness as we have not been able to find these exact statements in the literature.

\begin{lemma}\label{l:Falconer}
Suppose that for each $i\in\N$ we have a finite set $Q_i$ such that
$\diam Q_i\le d_i$, $Q_i$ is $\de_i$-separated
($x,y\in Q_i,x\neq y\Rightarrow |x-y|\ge\de_i$),
$|Q_i|=l_i$, $\sum_{i=1}^\infty \min Q_i>-\infty$ and
$\sum_{i=1}^\infty \max Q_i<\infty$. Let
$$
P=Q_1+Q_2+\ldots=\left\{\sum_{i=1}^\infty q_i\ :\ q_i\in Q_i\right\}.
$$

(a) If for for some $c<1$ we have $d_{i}\le cd_{i-1}$ for every $i\in\N$ then
$$
\ubdim P\le\limsup_{j\to\infty}
    \frac{\log(l_1\cdots l_{j})}
         {-\log d_{j}}.
$$

(b) If $d_i+\de_i\le \de_{i-1}$ for every $i\in\N$ then
$$
\hdim P\ge\liminf_{j\to\infty}
    \frac{\log(l_1\cdots l_{j})}
         {-\log(l_{j+1} \de_{j+1})}.
$$
%
\end{lemma}

\bp
The conditions $\sum_{i=1}^\infty \min Q_i>-\infty$ and
$\sum_{i=1}^\infty \max Q_i<\infty$ imply that the definition of $P$
makes sense and $P\su\R$ is bounded.
We will use the notation
$$
P_{q_1,\ldots,q_i}=q_1+\ldots+q_i+Q_{i+1}+Q_{i+2}+\ldots \quad
(q_1\in Q_1,\ldots,q_i\in Q_i).
$$

(a)
The condition $d_{i}\le cd_{i-1}$ implies that for some constant $C$ we
have $\sum_{j=i}^\infty d_j\le C d_i$ for any $i\in\N$.

Let $0<\de<\sum_{j=1}^\infty d_j$ be given. Choose $i\in\N$ so that
$\sum_{j=i+1}^\infty d_j < \de \le \sum_{j=i}^\infty d_j$.
Since $P$ is the union of $l_1\cdots l_i$ sets
of the form $P_{q_1,\ldots,q_i}$
($q_1\in Q_1,\ldots,q_i\in Q_i$) and
$\diam P_{q_1,\ldots,q_i}=\sum_{j=i+1}^\infty d_j < \de$,
we get that the minimal number of sets of diameter at most $\de$
that can cover $P$ is $N_{\de}\le l_1\cdots l_i$.
Thus
$$
\frac{\log N_{\de}}{-\log\de}\le
\frac{\log(l_1\cdots l_{i})}{-\log(\sum_{j=i}^\infty d_j)} \le
\frac{\log(l_1\cdots l_{i})}{-\log(C d_i)}.
$$
Therefore
$$
\ubdim P = \limsup_{\de\to 0}\frac{\log N_{\de}}{-\log\de}\le
   \limsup_{i\to\infty}  \frac{\log(l_1\cdots l_{i})}{-\log(C d_i)} =
\limsup_{i\to\infty} \frac{\log(l_1\cdots l_{i})}{-\log d_i}.
$$

(b) The proof is almost identical to the one in \cite[Example 4.6]{Falconer03}.
For each $i\in\N$ let $\mu_i$ be the equally distributed probability
measure on $Q_i$ and let $\mu$ be the product of these measures.
Thus
$$
\mu(P_{q_1,\ldots,q_i})=\frac1{l_1\cdots l_i} \quad
(q_1\in Q_1,\ldots,q_i\in Q_i).
$$

By translating each $Q_i$ we can move the minimum point of every $Q_i$
to $0$, so we can suppose that $\min Q_i=0$ and so $Q_i\su[0,d_i]$
for each $i$.

From the condition $d_i+\de_i\le\de_{i-1}$ by induction we get
$d_{i+1}+\ldots +d_{i+j}+\de_{i+j}\le \de_i$, so $d_{i+1}+d_{i+2}+\ldots \le \de_i$.
Hence $Q_{i+1}+Q_{i+2}+\ldots\su [0,\de_i)$.
This implies that the points
$\sum_{i=1}^\infty q_i$ ($q_i\in Q_i$) of $P$ are ordered
lexicographically; that is, $q_1=q'_1,\ldots,q_i=q_i',q_{i+1}<q'_{i+1}$
implies $\sum_{i=1}^\infty q_i<\sum_{i=1}^\infty q'_i$, and also that
for any fixed $i$, $P_{q_1,\ldots,q_i}$
($q_1\in Q_1,\ldots,q_i\in Q_i$) are pairwise disjoint sets of
diameter at most $\de_i$ and the
set of their leftmost points
is $\de_i$-separated.
This implies that an interval of length $h$ can intersect
at most $\floor{\frac{h}{\de_i}}+1$ of the form  $P_{q_1,\ldots,q_i}$.

Now let $U$ be an arbitrary subset of $\R$ with $\diam U=u<\de_1$.
By the mass distribution principle (see \cite[Mass Distribution Principle 4.2]{Falconer03}) it is enough
to show that $\mu(U)/u^s$ is bounded above by a constant if $s$ is less than
the righthand-side of the claimed inequality of (b).
Choose $j$ so that $\de_{j+1}\le u < \de_j$. By the last observation of
the previous paragraph, $U$ can intersect at most two sets of form
$P_{q_1,\ldots,q_{j}}$ and at most
$\lfloor\frac{u}{\de_{j+1}}\rfloor+1\le 2\frac{u}{\de_{j+1}}$ sets of form
$P_{q_1,\ldots,q_{j+1}}$.

This implies that
$$
\mu(U)\le\min\left(\frac2{l_1\cdots l_{j}},
                   \frac{2u/\de_{j+1}}{l_1\cdots l_{j+1}}\right)=
\frac{2}{l_1\cdots l_{j+1}}\min\left(l_{j+1},\frac{u}{\de_{j+1}}\right).
$$
Let $0<s<1$. Since $\min(l_{j+1},u/\de_{j+1})\le l_{j+1}^{1-s} (u/\de_{j+1})^s$
we get that
$$
\frac{\mu(U)}{u^s}\le\frac{2(l_{j+1} \de_{j+1})^{-s}}{l_1\cdots l_{j}},
$$
which is bounded above by a constant provided that
$$
s<\liminf_{j\to\infty}
    \frac{\log(l_1\cdots l_{j})}
         {-\log(l_{j+1} \de_{j+1} )}.
$$
\ep

The followig theorem completes the proof of Theorem \ref{t:vertices}, and also shows that Proposition~\ref{p:dimensions-1Dcenters}
and Corollary~\ref{c:dimensions-1Dcenters} are sharp.

\begin{theorem}\label{t:4examples}
(a) For any $s\in[0,2]$ there exist compact sets $B,S\su\R^2$
such that $\hdim(S)=\bdim(S)=\pdim(S)=s$,
$\pdim(B)=\bdim(B)=\frac{3s}{4}$ and
 every point of $S$ is the center of a square with all vertices in $B$.

(b) There exists a compact set $B\su\R^2$ such that
$\pdim(B)=\bdim(B)=\frac{3}{2}$ and
 every point of $[0,1]\times[0,1]$ is the center of a square with all vertices in $B$.

(c) For any $s\in[0,1]$ there exist compact sets $A\su\R$ and
$S\su\R^2$
such that $\hdim(S)=\bdim(S)=\pdim(S)=s$,
$\bdim(A)=\pdim(A)=\frac{3s}{8}$ and
for every $(x,y)\in S$ there exists $r$
such that $x-r,x+r,y-r,y+r\in A$.

(d) There exists a compact set $A\su\R$ such that
$\bdim(A)=\pdim(A)=\frac{3}{4}$ and
for every $(x,y)\in [0,1]\times[0,1]$ there exists $r$
such that $x-r,x+r,y-r,y+r\in A$.
\end{theorem}

\bp
Statements (a), (c) for $s=0$ are trivial, so we assume $s>0$. We prove all four claims of the theorem using the same construction.
For $k=1,2\ldots$ let $D_k$ be the set we obtain from Lemma~\ref{l:D_k},
$E_k=\{0,\ldots,k^4-1\}$ and
$\be_k=((k-1)!)^{-\frac8s}$.
Let
$$
A=\frac{\be_1}{1^4}\cdot D_1 + \frac{\be_2}{2^4}\cdot D_2+\ldots =
\left\{\sum_{k=1}^\infty \frac{\be_k}{k^4}a_k\ : \ a_k\in D_{k}\right\},
$$
$$
T=\frac{\be_1}{1^4}\cdot E_1 + \frac{\be_2}{2^4}\cdot E_2+\ldots =
\left\{\sum_{k=1}^\infty \frac{\be_k}{k^4}u_k\ : \ u_k\in E_{k}\right\},
$$
and set $B=A\times A$ and $S=T\times T$.

Let $(x,y)\in S$. Then $x=\sum_{k=1}^\infty \frac{\be_k}{k^4} u_k$ and
$y=\sum_{k=1}^\infty \frac{\be_k}{k^4} v_k$ for some
$u_k,v_k\in E_k=\{0,1,\ldots,k^4-1\}$.
For each $k$, by applying (\ref{everything}) of Lemma~\ref{l:D_k} to $D_{k}$
and $(u_k,v_k)$, we get $r_k\in\{1,\dots,k^4\}$ such that
$u_k-r_k,u_k+r_k,v_k-r_k,v_k+r_k\in D_{k}$.
Now let $r=\sum_{k=1}^\infty \frac{\be_k}{k^4} r_k$. Then $0<r<\infty$ and
$x-r,x+r,y-r,y+r\in A$, so $[x-r,x+r]\times[y-r,y+r]$ is a square
centered at $(x,y)$ and vertices in $B$.

It is clear that $S$, $A$ and $B$ are compact sets.
Note that if $s=2$ then $\be_k=((k-1)!)^{-4}$,
so $\frac{\be_k}{k^4}=((k-1)!)^{-4}/k^4=(k!)^{-4}=1^{-4}\cdots k^{-4}$,
hence $T=[0,1]$
and $S=[0,1]\times[0,1]$. Therefore, to complete the proof of all four
parts of the theorem it is enough
to show that the Hausdorff, packing and box dimensions of $S$
is $s$ and the packing and box dimensions of $A$ and $B$
are $3s/8$ and $3s/4$, respectively.

First we calculate the dimensions of $T$ by applying Lemma~\ref{l:Falconer}
to $Q_k=\frac{\be_k}{k^4} E_k$, $\de_k=\frac{\be_k}{k^4}$,
$d_k=\frac{\be_k}{k^4}(k^4-1)$ and $l_k=k^4$. Then $T=Q_1+Q_2+\ldots$.
We claim that $d_k+\de_k\le \de_{k-1}$ for any $k$. Indeed,
\[
d_k+\de_k=\frac{\be_k}{k^4}(k^4-1)+\frac{\be_k}{k^4}=\be_k=((k-1)!)^{-\frac8s},
\]
and so, using $s\le 2$, we get
\[
\de_{k-1}=\frac{\be_{k-1}}{(k-1)^4}=\frac{((k-2)!)^{-\frac8s}}{(k-1)^{4}}\ge
\frac{((k-2)!)^{-\frac8s}}{(k-1)^{\frac8s}} = ((k-1)!)^{-\frac8s} = d_k+\de_k.
\]
Since all other conditions of both parts of
Lemma~\ref{l:Falconer} are clearly satisfied, we can apply the lemma to get
\begin{align*}
\ubdim T &\le\limsup_{k\to\infty}
    \frac{\log(l_1\cdots l_{k})}
         {-\log d_{k}}\\
&\le \limsup_{k\to\infty} \frac{4\log(k!)}{-\log \be_{k}}\\
&=\limsup_{k\to\infty} \frac{4\log(k!)}{\frac8s\log ((k-1)!)}=\frac s2,\\
\\
\hdim T&\ge\liminf_{k\to\infty}
    \frac{\log(l_1\cdots l_{k})}
         {-\log (l_{k+1} \de_{k+1})}\\
         &=\liminf_{k\to\infty} \frac{4\log(k!)}{-\log \be_{k+1}}\\
         &=\liminf_{k\to\infty} \frac{4\log(k!)}{\frac8s\log (k!)}=\frac s2.
\end{align*}
Using the product formulas for upper box and Hausdorff dimension, we get that
\begin{align*}
\ubdim S&=\ubdim(T\times T)\le \ubdim T + \ubdim T\le \frac s2 +
\frac s2 = s,\\
\hdim S&=\hdim(T\times T)\ge \hdim T + \hdim T \ge \frac s2 +
\frac s2 = s.
\end{align*}
By the inequalities between dimensions, this implies that
$\hdim S=\pdim S=\bdim S=s$.

By applying Lemma~\ref{l:Falconer}(a) to $Q_k=\frac{\be_k}{k^4}D_k$,
$P=A$,
$\de_k=\frac{\be_k}{k^4}$, $d_k=3\be_k=3((k-1)!)^{-\frac8s}$ and
$l_k=|D_k|\le Ck^3$ we get
$$
\ubdim A\le\limsup_{k\to\infty}
    \frac{\log(l_1\cdots l_{k})}
         {-\log d_{k}}
\le \limsup_{k\to\infty} \frac{3\log(k!)+k\log C}{\frac8s\log ((k-1)!)+\log 3}
=\frac{3s}8.
$$
This implies that
$\ubdim B=\ubdim(A\times A)\le \ubdim A+\ubdim A\le \frac{3s}4$.

Combining the above results with Theorem~\ref{t:vertices}(b), we get that
\begin{align*}
\pdim B&\ge \frac34\pdim S=\frac34s\ge \ubdim B \ge \pdim B,\\
\lbdim B&\ge \frac34\lbdim S=\frac34s\ge \ubdim B \ge \lbdim B,
\end{align*}
thus
$\pdim B=\bdim B=\frac{3s}4$. Similarly, using
Proposition~\ref{p:dimensions-1Dcenters},
we get
\begin{align*}
\pdim A&\ge \frac38\pdim S=\frac38s\ge \ubdim A \ge \pdim A,\\
\lbdim A&\ge \frac38\lbdim S=\frac38s\ge \ubdim A \ge \lbdim A,
\end{align*}
thus
$\pdim A=\bdim A=\frac{3s}8$.
\ep

\begin{remark}
Note that in the proof $r\ge \sum_{k=1}^\infty 1\cdot \frac{\be_k}{k^4}$,
which is
a positive number that depends only on $s$. So in the above constructions
we do not need
to use small $r$ or small squares.
\end{remark}

As a corollary, we obtain the following strong version of Theorem \ref{t:packingbox-arbitrary-centers}(b').
\begin{corollary}\label{t:example-packing-boundary}
For any $s\in [0,2]$ there exist compact sets $S,B\su\R^2$
such that $\hdim(S)=\bdim(S)=\pdim(S)=s$,
$\bdim(B)=\pdim(B)=1+\frac{3s}{8}$ and
$B$ contains a square centered at
every point of $S$.
\end{corollary}
\bp
Let $A$ and $S$ be the sets from Theorem~\ref{t:4examples}(c),
and let $B=(A\times \R)\cup (\R\times A)$.
\ep

Likewise, we can easily complete the proof of Theorem \ref{t:packingbox-all-centers} by exhibiting the required example.

\bp[Proof of Theorem \ref{t:packingbox-all-centers}]
We have already proved Theorem \ref{t:packingbox-arbitrary-centers}(a),(b), which imply (taking $s=2$) that $\dim B\ge \tfrac{7}{4}$ for $\dim=\ubdim, \lbdim$ or $\pdim$. For the converse, let $A$ be the set from Theorem~\ref{t:4examples}(d), and set $B=(A\times \R)\cup (\R\times A)$.
\ep

\subsection{Countable constructions: the square boundary problem for box dimension}

We have now proved all the main results stated in the introduction, except for Theorem \ref{t:packingbox-arbitrary-centers}(a'). The construction here will be of a different kind: the sets $S, B$ will be countable, and obtained as the union of a sequence of discrete examples of size tending to zero and cardinality tending to infinity, at appropriate rates. These discrete sets are finitary analogs of the set $A$ from Theorem \ref{t:4examples}(d), and are described in the next lemma.

\begin{lemma} \label{l:discrete-dim}
 There exist sets $\{ A_N\}_{N\in\N}$ of natural numbers such that the following holds.
 \begin{enumerate}
  \item[(i)] For every $N$ and every $x,y\in\{0,\ldots,N-1\}$ there is $r\in\{1,\ldots,3N\}$ such that $x-r,x+r,y-r,y+r\in A_N$.
  \item[(ii)] For every $\delta>0$, there is $C=C(\delta)>0$ such that for every $R\in [1,N]$, the set $A_N$ can be covered by $C N^\delta (N/R)^{3/4}$ intervals of length $R$.
 \end{enumerate}
\end{lemma}
\bp
If we wanted the above only for $R=1$, then the sets $D_k$ from Lemma \ref{l:D_k} would suffice. As we will need large values of $R$ as well, we use a modification of the set $A$ from Theorem~\ref{t:4examples}(d) instead. Assume first that $N=(p!)^4$ for some $p\in\N$, and set
\[
 A_N = (p!)^4 \sum_{k=1}^p \frac{D_k}{(k!)^4},
\]
where $D_k$ is the set from Lemma \ref{l:D_k}. Claim (i) follows just like in the corresponding statement for $A$ in the proof of Theorem~\ref{t:4examples}.

For the second part, note that if $j\in\{1,\ldots, p\}$, then $A_N$ can be covered by
\[
|D_1|\cdots|D_j| = O(1)^j (j!)^3
\]
intervals of length
\[
(p!)^4 \sum_{k=j+1}^p \frac{(3k)^4}{(k!)^4}  \le \frac{200 (p!)^4}{(j!)^4}.
\]
This gives the claim (ii) in the case $R=R_j:=200 (p!/j!)^4$. We now consider a general $R$; we may assume that $R\in [200,N]$. Pick $j\in\{1,\ldots,p-1\}$ such that $R_{j+1}\le R\le R_j$. Then $A_N$ can be covered by $O(1)(N/R_{j+1})^{3/4}$ intervals of length $R_j$. Since $\log(j+1)!/\log j! \to 1$ as $j\to\infty$, for every $\delta>0$ there is $C>0$ such that $R_j < C\, R_{j+1}^{(1+\delta)}\le O(N^\delta)R_{j+1}$. We conclude that $A_N$ can be covered by $O(N^{2\delta})(N/R)^{3/4}$ intervals of length $R$, which yields the claim when $N=(p!)^4$.

The case of general $N$ follows in the same way, interpolating $N$ between consecutive values of $(p!)^4$ and using that $\log(p+1)!/\log p!\to 1$ as $p\to\infty$.
\ep

\begin{proof}[Proof of Theorem \ref{t:packingbox-all-centers}(a')]
The case $s=0$ is trivial, and we have already seen the case $s=2$. Hence we assume $s\in (0,2)$ and pick $\alpha>0$ such that $s=\frac{2\alpha}{1+\alpha}$.

For each $k\in\N$, let $N_k=\lfloor 2^{\alpha k}\rfloor$, and define
\begin{align*}
S_k &= \{ 0,\ldots, N_k-1\}\times \{ 0,\ldots, N_k-1 \},\\
B_k &= A_{N_k}\times [-3N_k,4N_k] \cup [-3N_k,4 N_k]\times A_{N_k},
\end{align*}
where $A_N$ are the sets given by Lemma \ref{l:discrete-dim}.  Let $\eps_k=2^{-(1+\alpha)k}$, and define
\begin{align*}
S &= \{0\}\cup \bigcup_{k=1}^\infty (2^{-k},0)+ \eps_k S_k,\\
B &= C_0 \cup \bigcup_{k=1}^\infty (2^{-k},0)+ \eps_k B_k,
\end{align*}
where $C_0$ is the boundary of a square of unit side-length centered at the origin. The sets $S,B$ are clearly compact.

Firstly, it follows from Lemma \ref{l:discrete-dim}(i) that $B_k$ contains a square boundary with center in every point of $S_k$, and hence $B$ contains a square boundary with center in every point of $S$.

In light of Theorem~\ref{t:packingbox-arbitrary-centers}(a), we only need to show that $\lbdim S\ge s$ and $\ubdim B\le \max(1,7s/8)+O(\delta)$ where $\delta>0$ is arbitrary.  For the first part, note that since $S$ contains a translate of $\eps_k S_k$, it contains $|N_k|^2=\Omega(2^{2\alpha k})=\Omega(\eps_k^{-s})$ points at pairwise distance at least $\eps_k$. Interpolating an arbitrary $\eps\in (0,1)$ between consecutive values of $\eps_k$, we deduce that $\lbdim S \ge s$.

Let us now count how many squares of side-length $\eps_k$ are required to cover $B$. Split $B=B'_k\cup B''_k\cup B'''_k$, where
\begin{align*}
B'_k &= C_0 \cup \bigcup_{j=1}^{k-1} (2^{-j},0)+ \eps_j B_j, \\
B''_k &= \bigcup_{j:\eps_j\le \eps_k\le N_j\eps_j} (2^{-j},0)+ \eps_j B_j,\\
B'''_k &= \bigcup_{j:N_j\eps_j<\eps_k} (2^{-j},0)+ \eps_j B_j.
\end{align*}
For $j<k$, the set $\eps_j B_j$ consists of $2|A_{N_j}|$ segments of length $O(\eps_j N_j)$. Since $|A_{N_j}|=O(2^{3\alpha j/4+\alpha\delta})$ by Lemma \ref{l:discrete-dim} applied to $R=1$, $\eps_j B_j$ can be covered by
\[
O(1)2^{3\alpha j/4+\alpha\delta}\frac{\eps_j N_j}{\eps_k} = O(1) 2^{(3\alpha/4-1+\alpha\delta)j}\eps_k^{-1}
\]
balls of radius $\eps_k$. Note that $3\alpha/4-1<0$ if and only if $s<8/7$, and in this case (taking $\delta$ small enough depending on $s$) $B'_k$  can be covered by  $O(\eps_k^{-1})$ balls of radius $\eps_k$. Otherwise, if $s\ge 8/7$, then $B'_k$ can be covered by
\[
O(1) k 2^{(3\alpha/4-1+\alpha\delta)k}\eps_k^{-1}= k 2^{\alpha\delta k} O(2^{7\alpha k/4}) = O(\eps_k^{-7s/8-O(\delta)})
\]
balls of radius $\eps_k$.

We now look at $B''_k$. Hence, suppose $\eps_j\le\eps_k\le N_j\eps_j$, and let $R\in [1,N_j]$. By the second part of Lemma \ref{l:discrete-dim}, the set $A_{N_j}$ can be covered by $O(1)N_j^\delta(N_j/R)^{3/4}$ intervals of length $R$, so it follows that $B_j$ can be covered by $O(N_j^\delta)O(N_j/R)(N_j/R)^{3/4}$ balls of radius $R$, and therefore $\eps_j B_j$ can be covered by $O(N_j^\delta)(N_j/R)^{7/4}$ balls of radius $\eps_j R$. We apply this to $R=\eps_k/\eps_j$, which is indeed in $[1,N_j]$, and deduce that $\eps_j B_j$ can be covered by
\[
O(N_j^\delta)\left(\frac{N_j \eps_j}{\eps_k}\right)^{7/4} = O(1) \left(2^{(1-O(\delta))j}\eps_k\right)^{-7/4}
\]
balls of radius $\eps_k$. By taking $\delta$ small enoguh and adding over $j\ge k$, it follows that $B''_k$ can be covered by
\[
O(1) \left(2^{(1-O(\delta))k}\eps_k\right)^{-7/4} = O(1) 2^{(7\alpha/4+O(\delta))k} = O(1) \eps_k^{-7s/8-O(\delta)}
\]
balls of radius $\eps_k$.

Finally, the smallest $j$ such that $N_j \eps_j < \eps_k$ satisfies $2^{-j}=O(\eps_k)$, whence $B'''_k$ has diameter $O(\eps_k)$ and thus can be covered by $O(1)$ balls of radius $\eps_k$.

We conclude that $B$ can be covered by $O(1)\max(\eps_k^{-1})$ balls of radius $\eps_k$ if $s<8/7$, and by $O(1)\eps_k^{-7s/8-O(\delta)}$ balls of radius $\eps_k$ if $s\ge8/7$. This implies that $\ubdim B\le \max(1,7s/8+O(\delta))$, which is what we wanted to show.
\ep

\section{Concluding remarks and open questions} \label{sec:remarks}

Our results suggest a number of problems and future directions of work. We finish by briefly discussing some of these.

All Kakeya-type problems have an associated maximal operator, defined (using the notation from the introduction) as $M f(x)= \sup_{C\in\mathcal{F}_x} A(f,C)$, where $f$ is a continuous function on the corresponding space, and $A(f,C)$ is the average of $f$ over the set $C$ (defined using the natural measure on these sets). One is interested in bounding these operators in $L^p$ for suitable values of $p$, or appropriate discretizations when they are unbounded, and this typically implies a lower bound on the size of the packing sets $B$ as a corollary (see e.g. \cite{Wolff99} for some examples). In our setting, the first guess for maximal operator might be
\[
 M_{\eps}f(x) = \sup_{r>0} \frac{1}{4\eps r} \int_{S_\eps(x,r)} f \,dx,
\]
where $S_\eps(x,r)$ is the $\eps$-neighborhood of the square $S(x,r)$ with center $x$ and side length $r$. However, if we take
$f$ as
the indicator of a small neighborhood of $\ell=[-1,2]\times \{0\}$, then we see that $Mf(x) \ge \tfrac{1}{4}$ for all $x\in [0,1]^2$ (just consider the square with one side inside $\ell$). In other words, we are obtaining a constant lower bound for a function with very small norm from just one side of the squares involved; this leads to unnatural (and trivial) results, analog to the ones for Hausdorff dimension. A more interesting operator would then be
\[
 M_{\eps}f(x) = \sup_{r>0} \min_{j=1}^4 \frac{1}{\eps r} \int_{S_\eps^{(j)}(x,r)} f \,dx
\]
where $S^{(j)}_\eps(x,r)$ are the $\eps$-neighborhoods of the sides of the square  $S(x,r)$. This maximal operator forces us to take into consideration all four sides of the square, but it is somewhat awkward, as it is not sub-linear. Nevertheless, it would be interesting to understand its behavior.

Another natural direction to pursue is to allow squares with arbitrary orientations. We have seen in Proposition \ref{p:2D-hausdorff-rotated} that, at least for some of the problems involving vertices of a square, allowing rotations does change the answer. Nevertheless, it seems geometrically clear that for problems involving the whole square boundaries, one obtains maximal overlap between squares if they all share the same orientation, so the optimal configuration should be of this form and the results should not change if rotations are allowed. So far we have not been able to prove this rigorously, although it is not hard to obtain lower bounds which show that the lower box counting dimension in the case when there are centers in all points of $[0,1]^2$ still has to be much larger than $1$.

A final set of questions concerns what happens if squares are replaced by other polygons or by cubes or polyhedra. Theorem \ref{t:Hausdorff} can be seen to hold for arbitrary polygons with a similar proof, but we do not know much for the other problems. A random construction can be used to show that for any
$n$-gon
$P$, there exists a set $B\subset\R^2$ with $\hdim B= \ubdim B= 2-\tfrac{1}{n}$ which 
contains a homothetic copy of $P$ with center in every point of $[0,1]^2$, but we know little about the opposite inequality. For triangles, some results are easier; for example, if a set $B$ contains a homothetic copy of a fixed triangle with center in any point of another set $S$, then $\pdim B\ge 1+\tfrac{1}{3}\pdim S$. The idea is to consider the sets $A_j$, $j=1,2,3$ that parametrize the lines containing the corresponding sides of the triangles that make up $B$, and observe that one can recover $S$ from $A_1\times A_2\times  A_3$ (more precisely, there is a Lipschitz map $A_1\times A_2\times A_3\to S$, given by mapping the triple of lines to the center of the triangle they determine). This implies that for some $j$, $\pdim A_j\ge \tfrac{1}{3}\pdim S$ and the result follows in a similar way to the proof of Theorem \ref{t:packingbox-arbitrary-centers}(b).


\end{document}